\documentclass[leqno,11pt]{amsart}

\usepackage{amssymb}
\usepackage{amsmath}
\usepackage{enumerate}
\usepackage[pdftex]{graphicx}
\usepackage{graphicx, color}

\def\s{\mathbb{S}}
\def\N{\mathbb{N}}
\def\R{\mathbb{R}}

\def\C{\mathbb{C}}

\newtheorem{remark}{Remark}[section]
\newtheorem{theorem}{Theorem}[section]
\newtheorem{proposition}{Proposition}[section]
\newtheorem{corollary}{Corollary}[section]
\newtheorem{lemma}{Lemma}[section]
\newtheorem{example}{Example}[section]

\numberwithin{equation}{section}

\begin{document}

\title[Homothetic solitons for the IMCF]{Homothetic solitons \\ for the inverse mean curvature flow}

\author{Ildefonso Castro}
\address{Departamento de Matem\'{a}ticas \\
Universidad de Ja\'{e}n \\
23071 Ja\'{e}n, SPAIN} \email{icastro@ujaen.es}

\author{Ana M.~Lerma}
\address{Departamento de Did\'{a}ctica de las Ciencias \\
Universidad de Ja\'{e}n \\
23071 Ja\'{e}n, SPAIN} \email{alerma@ujaen.es}

\thanks{Research partially supported by a MEC-Feder grant
MTM2014-52368-P}



\date{}

\begin{abstract}
We study solutions to the inverse mean curvature flow which evolve by homotheties of a given submanifold with arbitrary dimension and codimension. We first show that the closed ones are necessarily spherical minimal immersions and so we reveal the strong rigidity of the Clifford torus in this setting. Mainly we focus on the Lagrangian case, obtaining numerous examples and uniqueness results for some products of circles and spheres that generalize the Clifford torus to arbitrary dimension. We also characterize the pseudoumbilical ones in terms of soliton curves for the inverse curve shortening flow and minimal Legendrian immersions in odd-dimensional spheres. As a consequence, we classify the rotationally invariant Lagrangian homothetic solitons for the inverse mean curvature flow.
\end{abstract}

\maketitle

%%%%%%%%%%%%%%%%%%%%%%%%%%%%%%%%%%%%%%%%%%%%%%%%%%%%%%%%%%%%%%%%%%%%%%%%%%%%%%%%%%%%%%%%%%%%%%%%%%%%%%%%%%%%%%%%%

\section{Introduction}

A smooth family of immersions $F=F(\cdot,t)$ into Euclidean space evolves by the inverse mean curvature flow if
\begin{equation*} %\label{IMCF} 
\left( \frac{\partial F}{\partial t}\right)^\bot = -\frac{H}{|H|^2} %\tag{IMCF}
\end{equation*}
where $H$ is the mean curvature vector of $F$ and $^\bot$ denotes the normal component. 
There are many interesting papers on the inverse mean curvature flow (shortly IMCF) in the case of hypersurfaces; see for example \cite{Ge90}, \cite{HI08}, \cite{Sm90} and \cite{Ur90}. We emphasize two very important applications of the IMCF: on the one hand, Huisken and Ilmanen (\cite{HI97}, \cite{HI01}) used a level set approach for the IMCF and were able to prove the Riemannian Penrose inequality; on the other hand, using IMCF Bray and Neves (\cite{BN04}) proved the Poincar\'{e} conjecture for 3-manifolds with Yamabe invariant greater than real projective 3-space.

Round hyperspheres in Euclidean space expand under the IMCF with an exponentially increasing radius. When we look for the condition on an immersion $\phi : M^n \rightarrow \R^m$ evolving by homotheties under the IMCF, we get that the evolution is given necessarily by $F(x,t)=e^{a t}\phi (x), t\geq 0$, for some $ a\neq 0$, and $\phi$ has to verify  the following equation:
\begin{equation}\label{IMCFsoliton}
-\frac{H}{|H|^2}=a\, \phi^\bot , \  a\neq 0 \tag{*}
\end{equation}
where $H={\rm trace}\,\sigma$ is the nonvanishing mean curvature vector of the immersion $\phi$ ($\sigma$ is the second fundamental form) and $^\bot$ denotes the projection onto the normal bundle of $M$. 
The solutions of (\ref{IMCFsoliton}) are called {\em homothetic solitons for the inverse mean curvature flow with constant velocity $a\neq 0$}.
We remark that (\ref{IMCFsoliton}) is invariant by dilations of $\phi$.
Hence the solutions of (\ref{IMCFsoliton}) give rise to homothetically shrinking or expanding solutions of the inverse mean curvature flow  according to $a<0$ or $a>0$. In this way, solutions of (\ref{IMCFsoliton}) with $a<0$
(resp.\ $a>0$) are simply called {\em shrinkers} (resp.\ {\em expanders}) {\em for the IMCF.}

When $n=1$ and $m=2$, the solutions of (\ref{IMCFsoliton}) are explicitly described in \cite{DLW15}. The corresponding
plane curves include the classical logarithmic spirals and involutes of circles as expanders and the only closed ones
are the circles centered at the origin (\cite{And03}).
When $m=n+1$ and $n\geq 2$, it is also proved in \cite{DLW15} that the hyperspheres centered at the origin are
exceptionally rigid since they are the only closed homothetic soliton hypersurfaces for the inverse mean curvature flow. However, in higher codimension, Drugan, Lee and Wheeler also observed in \cite{DLW15} that any minimal submanifold of the standard hypersphere is an expander for the IMCF.

We show in Theorem~\ref{closed} that the closed homothetic solitons for the inverse mean curvature flow are necessarily spherical minimal submanifolds. Making use of this fact, we also reveal a strong rigidity for the Clifford torus $\s^1 \times \s^1$ in the class of closed homothetic solitons for the IMCF by showing that it is the only embedded torus in $\R^4$ (Corollary~\ref{CliffordBrendle}), the only (besides the two-sphere) with constant Gauss curvature in $\R^4$ (Corollary~\ref{CliffordLawson}) and, surprisingly, the only equality case for a pinching on $|\sigma|^2/|H|^2$ in general dimension with equal codimension (Corollary~\ref{Cor IMRN14}).

If we pay attention to the Lagrangian homothetic solitons for the IMCF in $\R^{2n}\equiv\C^n$ ($n\geq 2$), 
we show in Theorem~\ref{ClosedLagr} that the closed ones are specifically Hopf immersions over minimal Lagrangian compact submanifolds in the complex projective space $\C P^{n-1}$. In this way we first intensify the strong rigidity of the Clifford torus in this setting proving it is the only closed Lagrangian in $\C^2$ (Corollary~\ref{ClosedLagrDim2}). Moreover, we provide several uniqueness results (see Corollaries \ref{ClosedLagrKDim3}, \ref{ClosedLagrKcte}, \ref{ClosedLagrPinching} and \ref{ClosedLagrSO(n)}) for the product of $n$ circles 
$\s^1 \times \dots \times \s^1 \hookrightarrow \C^n$ and the Hopf immersion
$\s^1 \times \s^{n-1} \rightarrow \C^n$, $ (e^{it},(x_1,\dots,x_n)) \mapsto  \sqrt n e^{it}\,(x_1,\dots,x_n)$, that generalize in two different ways the Clifford torus $\s^1 \times \s^1 \subset \C^2$ to arbitrary dimension
via Lagrangian immersions.

In the last section of the paper we confirm that the class of non compact Lagrangian homothetic solitons for the IMCF is also huge. They include not only product submanifolds  $\s^1 \times \stackrel{k)}{\dots} \times \,\s^1 \!\times \R^{n-k} \hookrightarrow \C^{n}$, $1\leq k \leq n-1$, with constant velocity $a=1/k$, but also generalized cylinders for arbitrary $a\neq 0$ (see (\ref{IMCFsoliton})) of type $\mathcal C \times \R^{n-1}\hookrightarrow \C^{n}$, where $\mathcal C$ is a homothetic soliton for the inverse curve shortening flow  with $a\neq 0$ the constant velocity of $\mathcal C$. Moreover, we characterize in Theorem~\ref{LagrPseudo} the pseudoumbilical ones (cf.\ \cite{CMU01} and \cite{RU98}) in terms of powers of soliton curves for the inverse curve shortening flow and minimal Legendrian immersions in odd-dimensional spheres. As a consequence, we classify in Corollary \ref{LagrRev} the rotationally invariant Lagrangian homothetic solitons for the IMCF.

\vspace{0.3cm}

\section{Solitons for the inverse mean curvature flow}

Let $\phi : M^n \rightarrow \R^m$ be an isometric immersion of an $n$-dimensional manifold in Euclidean $m$-space with
nonvanishing mean curvature vector $H$. We consider $H={\rm trace}\,\sigma$, where $\sigma $ is the second
fundamental form of $\phi $. The Euclidean metric in $\R^m$ will be denoted by $\langle \cdot, \cdot \rangle$. The submanifold $M$ in $\R^m$ is called a {\em homothetic soliton for the inverse mean
curvature flow} if there exists a constant $a\neq 0$ satisfying
\begin{equation}\label{IMCFsol}
-\frac{H}{|H|^2}=a\, \phi^\bot
\end{equation}
where $^\bot$ stands for the projection onto the normal bundle of $M$. This definition obeys the fact that
\[ F(x,t)=e^{a t}\phi (x), t\geq 0 \]
is then a solution to the inverse mean curvature flow (IMCF for short) given by
\[ \left( \frac{\partial F}{\partial t}\right)^\bot = -\frac{H}{|H|^2} .  \]
Thus $a$ is said to be the {\em velocity constant} of $\phi$. When $a>0$ the soliton is simply called an {\em expander
for the IMCF} and when $a<0$ a {\em shrinker for the IMCF}.

It is remarkable that (\ref{IMCFsol}) is invariant by homotheties; that is, if $\phi $ is a homothetic soliton for the IMCF with velocity constant $a$, for any $\rho >0$ the rescaled immersion $\rho\,  \phi$ is a
homothetic soliton for the IMCF with the same velocity constant $a$.
For this reason, the possible classifications results on homothetic solitons for the IMCF must be done modulo dilations and the necessary hypothesis must be invariant by homotheties.

The solutions of (\ref{IMCFsol}) when $n=1$ and $m=2$ are known as homothetic solitons for the inverse curve shortening flow. They are explicitly described in \cite{DLW15} and include the classical logarithmic spirals and involutes of circles as expanders; the only closed ones are the circles centered at the origin (\cite{And03}).
When $n\geq 2$, the simplest examples of expander hypersurfaces for the IMCF are (modulo homotheties) the standard sphere $\s^n \subset
\R^{n+1}$ ($a=1/n$) and the cylinders $\s^k \times \R^{n-k} \subset \R^{n+1}$ ($a=1/k$), $1\leq k \leq n-1$. 

\begin{remark}\label{First examples}
{\rm Spherical minimal submanifolds $\phi:M^n \rightarrow \s^{m-1}(R) \subset \R^m$ satisfy $H=-(n/R^2)\phi$ and then it is clear that they are expanders for the IMCF with $a=1/n$. We notice that $|H| |\phi|=n$ in this case.}
\end{remark}

In higher
codimension we make use of Remark~\ref{First examples} to point out, up to homotheties, the following compact examples:
\begin{example}\label{Cliford}
For any $n_1,n_2\in \N$ such that $n_1+n_2=n$, the Clifford immersion
\[
\s^{n_1}({\sqrt{n_1} }) \times \s^{n_2}({\sqrt{n_2} }) \hookrightarrow \R^{n+2}
\]
is a compact expander for the IMCF ($a=1/n$) in $\R^{n+2}$ with $|\sigma|^2 / |H|^2 =2/n$.
\end{example}
\begin{example}\label{ProductCircles}
The product of $n$ circles
\[
\s^1 \times \stackrel{n)}{\dots} \times \s^1 \hookrightarrow \R^{2n}
\]
is a compact flat expander for the IMCF ($a=1/n$) with $|\sigma|^2 / |H|^2 =1 $ in $\R^{2n}$.
\end{example}
\begin{example}\label{Anciaux0}
The Hopf immersion
\[
\s^1 \times \s^{n-1} \rightarrow \C^n\equiv\R^{2n}, \quad (e^{it},(x_1,\dots,x_n)) \mapsto  \sqrt n \, e^{it}\,(x_1,\dots,x_n)
\]
is a compact expander for the IMCF ($a=1/n$) with $|\sigma|^2 / |H|^2 = (3n-2)/n^2$ in $\R^{2n}\equiv\C^n$.
\end{example}
We notice that the Example~\ref{Cliford}, Example~\ref{ProductCircles} and Example~\ref{Anciaux0} coincide when
$n=2$ providing precisely the well-known {\em Clifford torus} $\,\s^1 \times \s^1$ in $\R^4$. 

We deduce the following geometric properties from the definition of homothetic soliton for the IMCF.
\begin{lemma}\label{properties}
Let $\phi : M^n \rightarrow \R^m$ be a homothetic soliton for the inverse mean curvature flow with constant velocity
$a$. Then:
\begin{enumerate}
\item[(a)] $\langle H, \phi \rangle =-1/a$.
\item[(b)] $\triangle |\phi|^2 =2 (n-1/a)$.
\end{enumerate}
\end{lemma}
\begin{proof}
Using (\ref{IMCFsol}) we have that $|H|^2=-a|H|^2 \langle \phi , H \rangle$ and part (a) follows recalling that $H$ is
nonvanishing.  On the other hand, using that $H={\rm trace}\,\sigma$, an easy computation gives us that the Laplacian
of the squared norm of $\phi$ is given by $\triangle |\phi|^2 =2 (\langle H, \phi \rangle+n)$ and, using (a), we
conclude part (b).
\end{proof}
\begin{remark}\label{iff}
{\rm It is proved in Proposition 3 of \cite{DLW15} that condition (a) in Lemma~\ref{properties} characterizes the homothetic solitons for the IMCF of codimension one, including plane curves.}
\end{remark}

As a consequence of Lemma~\ref{properties}, we get the following characterization of the closed homothetic solitons for the IMCF (see
Remark~\ref{First examples}).
\begin{theorem}\label{closed}
Let $\phi : M^n \rightarrow \R^m$ be a homothetic soliton for the inverse mean curvature flow with constant velocity
$a$. If $M$ is closed, then $a=1/n$ (i.e.\ $\phi$ is an expander for the IMCF) and $\phi$ must be a spherical minimal immersion, that is, $M $ is a minimal submanifold in some $(m-1)$-sphere (centered at the origin). 
\end{theorem}
\begin{proof}
We use Lemma~\ref{properties} when $M$ is closed. By integrating (b) we easily obtain that $a=1/n$ and, in
addition, $|\phi|^2$ is harmonic, hence constant, say $|\phi|\equiv R$. So we have that $\phi : M^n \rightarrow \s^{m-1}(R)\subset \R^m$ is a spherical immersion with mean curvature vector $\hat H$. Then $H=\hat H -(n/R^2)\phi$. It is also clear that $\phi^\bot=\phi$ in this case: in general, $\phi ^\top= \frac{1}{2}\nabla |\phi |^2$ (where $\nabla $ means gradient with respect to the induced metric on $M$ and $^\top $ denotes projection onto the tangent bundle) so that $\phi ^\top =0$ since $|\phi |$ is constant. As $\phi $ is an
expander for the IMCF we deduce that $\hat H =0$, that is, $M $ is a minimal submanifold in the $(m-1)$-sphere of
radius $R$ centered at the origin.
\end{proof}

Example~\ref{ProductCircles} and Example~\ref{Anciaux0} belong to a special class of submanifolds of codimension $n$:
the {\em Lagrangian submanifolds}. An immersion $\phi :M^n \rightarrow \R^{2n}\equiv \C^n$ is said to be Lagrangian if the
restriction to $M$ of the Kaehler two-form $\omega (\,\cdot\, ,\,\cdot\,)=\langle J\cdot,\cdot\rangle$ of $\C^n$
vanishes. Here $J$ is the complex structure on $\mathbb{C}^n$ that defines a bundle isomorphism between the tangent and
the normal bundle of $\phi $ for a Lagrangian submanifold.  On the other hand, the mean curvature vector of a
Lagrangian submanifold can be written as 
\begin{equation}\label{Hlagr}
H=J\nabla \theta 
\end{equation} 
where $\theta:M\rightarrow \mathbb{R} /2\pi\mathbb{Z}$
is the so-called Lagrangian angle map of $\phi$.
%In general $\theta $ is a multivalued function; nevertheless
%$\alpha_H=-d\theta = \langle JH,\cdot\rangle$ is a well defined closed 1-form on $M$ and its cohomology class
%$[\alpha_H]$ is called the Maslov class of $\phi$.

\vspace{0.3cm}

\section{Closed expanders for the inverse mean curvature flow}

In this section we make use of Theorem~\ref{closed} to get some uniqueness results in the class of the closed homothetic solitons for the IMCF described in the previous section.  The first direct application is considering $m=n+1$, which implies that $M^n$ would be a closed (i.e.\ compact without boundary) submanifold in $\s^n(R)$. In this way, we deduce the strong rigidity of the spheres in this setting, first proved in Theorem~12 of \cite{DLW15}:

\begin{corollary}\label{spheres}
Let $M^n$ be a homothetic soliton hypersurface for the inverse mean curvature flow in $\R^{n+1}$ ($n\geq 2$). If $M$ is closed then $M$ is, up to dilations, the round hypersphere $\s^n$ (centered at the origin).
\end{corollary}

The second direct application, also commented in \cite{DLW15}, is deduced thanks to the Lawson's construction (cf.\
\cite{La70}) of minimal surfaces in $\s^3$: {\em For any integer $g\geq 1$, there exists at least one two-dimensional
compact embedded expander for the IMCF of genus $g$ in $\R^4$}. The Brendel proof of Lawson's conjecture (cf.\ \cite{Br13})
translates into the following characterization of the Clifford torus $\s^1 \times \s^1$.

\begin{corollary}\label{CliffordBrendle}
Let $M^2$ be a homothetic soliton for the inverse mean curvature flow in $\R^4$. If $M$ is an embedded torus then $M$ is the Clifford torus.
\end{corollary}

The third direct application comes from the rigidity results in Corollary~3 and Theorem~3 of \cite{La69} with assumptions on the Gauss curvature of the surface.

\begin{corollary}\label{CliffordLawson}
Let $M^2$ be a closed homothetic soliton for the inverse mean curvature flow in $\R^4$. If the Gauss curvature $K$ of $M$ is constant or non negative then $M$ is, up to dilations, the sphere $\s^2$ or the Clifford torus $\s^1 \times \s^1$.
\end{corollary}

But we can get another characterization of the sphere, the Clifford torus and the Veronese immersion (in a not so
trivial way) with hypothesis involving a pinching condition on the norms of the second fundamental form and the mean
curvature vector, condition that is invariant by dilations as just happens to the definition of homothetic soliton for
the IMCF.

\begin{corollary}\label{Th IMRN14}
Let $\phi : M^n \rightarrow \R^{n+p}$ ($n\geq 2$, $p\geq 1$) be a closed homothetic soliton for the inverse mean
curvature flow. If
\begin{equation}\label{boundA}
\frac{|\sigma|^2}{|H|^2}\leq \frac{3p-4}{n(2p-3)}
\end{equation}
then:
\begin{enumerate}
\item either $|\sigma|^2/|H|^2= 1/n$ and $M$ is, up to dilations, the hypersphere $\,\s^n$ (centered at the origin) in
$\R^{n+1}$ (i.e. $p=1$),
\item or $|\sigma|^2 / |H|^2= \frac{3p-4}{n(2p-3)}$ and $M$ is, up to dilations,
\begin{enumerate}
\item either the Clifford immersion of $\,\s^{n_1}(\sqrt{n_1})\times\s^{n_2}(\sqrt{n_2})$, $n_1+n_2=n$,  in $\R^{n+2}$ (with $|\sigma|^2 / |H|^2= 2/n$, i.e.\ $p=2$),
\item or the Veronese immersion  of
$\,\s^2$ in $\R^5$  (with $|\sigma|^2 / |H|^2= 5/6$, i.e.\ $n=2$, $p=3$).
\end{enumerate}
\end{enumerate}
\end{corollary}

In a similar way to Theorem A in \cite{CL14}, the proof of Corollary~\ref{Th IMRN14} consists in translating the well-known results of Simon \cite{Si68}, Lawson
\cite{La69} and Chern-Do Carmo-Kobayashi \cite{CdCK78} about intrinsic rigidity for minimal submanifolds in the unit sphere. Concretely, one must simply take into account that if $\phi : M^n \rightarrow
\s^{m-1}(R)\subset \R^m$ is a spherical minimal immersion with second fundamental form $\hat \sigma $, then
$H=-(n/R^2)\phi$ (and so $|H||\phi|=n$) and, in addition, it satisfies $ |\sigma|^2=|\hat\sigma|^2 +n / R^2 $.

If the dimension and the codimension of the submanifold $M$ coincide (what happens, for example, when $M$ is
Lagrangian), only the case {\it (2)-(a)} with $n=2$ in Corollary~\ref{Th IMRN14} is possible; this immediately implies the following surprising characterization of the Clifford
torus $\s^1 \times \s^1 $.

\begin{corollary}\label{Cor IMRN14}
Let $\phi : M^n \rightarrow \R^{2n}$ be a closed homothetic soliton for the inverse mean curvature flow with
codimension $n\geq 2$. If
\begin{equation}\label{boundAn}
\frac{|\sigma|^2}{|H|^2}\leq \frac{3n-4}{n(2n-3)}
\end{equation}
then $n=2$, $|\sigma|^2 = |H|^2$ and $M$ is, up to dilations, the Clifford torus $\, \s^1 \times \s^1 $ in $\,
\R^4$.
\end{corollary}

We emphasize that the Example~\ref{Cliford}, Example~\ref{ProductCircles} and Example~\ref{Anciaux0} coincide when $n=2$ providing precisely the Clifford torus $\s^1 \times \s^1$ in $\R^4$. In fact, we remark that
Example~\ref{ProductCircles} and Example~\ref{Anciaux0} satisfy the hypothesis (\ref{boundAn}) of Corollary~\ref{Cor IMRN14} only when $n=2$.
\vspace{0.1cm}

In the Lagrangian setting we show that there are many examples of closed expanders for the IMCF thanks to the following result.

\begin{theorem}\label{ClosedLagr}
Let $\phi : M^n \rightarrow \R^{2n}\equiv \C^n$ be a closed Lagrangian homothetic soliton for the inverse mean
curvature flow. Then, up to dilations, $M=\s^1 \times N^{n-1}$, where $N$ is a closed $(n\!-\!1)$-manifold, and $\phi$ is given by
\[
\phi(e^{it},x)=e^{it} \psi (x)
\]
where $\psi:N^{n-1}\rightarrow\s^{2n-1}\subset \C^n$ is a minimal Legendrian immersion.
\end{theorem}
\begin{proof}
Using that $M$ is closed and the description in \cite{Ch13} of the spherical Lagrangian
submanifolds in complex Euclidean space, we deduce that $M=\s^1 \times N^{n-1}$, where $N$ is also closed, and  $\phi$ is given, up to dilations, by $\phi(e^{it},x)=e^{it} \psi (x)$, where $\psi:N^{n-1}\rightarrow\s^{2n-1}\subset \C^n$ is a Legendrian immersion. It is easy to check that $\phi^\perp= \phi =-J\phi_t $. 
Using the formula for $H$ given in part 4 of Proposition~3 in \cite{RU98}, we deduce that $\psi$ is necessarily minimal in order to $H$ be collinear to $\phi^\perp $ and, in addition, $H=n J\phi_t$ and so $\phi$ satisfies (\ref{IMCFsol}).
\end{proof}
\begin{remark}\label{LegenLagr}
{\rm There are many examples of compact minimal Legendrian submanifolds in odd-dimensional spheres (see \cite{Na81}, Section 2 in \cite{BG04} or Section 2 in \cite{CU04} and Section 4 in \cite{CLU06}). Moreover, 
minimal Legendrian submanifolds in $\s^{2n-1}$ are in
correspondence to minimal Lagrangian submanifolds in the complex projective space $\C P^{n-1}$ projecting via the Hopf fibration $\Pi:\s^{2n-1} \rightarrow \C P^{n-1}$. Hence in Theorem~\ref{ClosedLagr} we arrive at a family of closed Lagrangian submanifolds in $\C^n $ which can be considered Hopf immersions over closed minimal Lagrangian submanifolds in $\C P ^{n-1}$ (see the Hopf tori studied in \cite{Pi85}).

The Hopf immersion described in the Example~\ref{Anciaux0} corresponds (up to a homothety) to the election of the totally geodesic Legendrian immersion $\psi:\s^{n-1}\rightarrow\s^{2n-1}\subset \C^n$, $\psi (x)=x$, 
which produces the totally geodesic Lagrangian $\s^{n-1}$ in $\C P^{n-1}$. 

The Example~\ref{ProductCircles} is produced (up to a homothety of ratio $\sqrt n$) by considering the minimal Legendrian immersion given by
\begin{eqnarray} 
\R^{n-1} & \rightarrow & \s^{2n-1}\subset \C^n \nonumber \\
(x_1,\dots,x_{n-1}) & \mapsto  &
\frac{1}{\sqrt{n}} \left(e^{ix_1},\dots,e^{ix_{n-1}}, e^{-i(x_1+ \dots +x_{n-1})} \right)  \nonumber
\end{eqnarray}
which provides the torus $T^{n-1}$ described in Section III.3.A of \cite{HL82}. In this way we get
a finite Riemannian covering of the so-called generalized Clifford torus in $\C P^{n-1}$ (see for instance \cite{LZ94}).
}
\end{remark}

In the two-dimensional case, since a minimal Legendrian curve in $\s^3$ must be totally geodesic, as an immediate consequence of Theorem~\ref{ClosedLagr} we get a new rigidity result for the Clifford torus.

\begin{corollary}\label{ClosedLagrDim2}
Let $\phi : M^2 \rightarrow \R^{4}$ be a closed homothetic soliton for the inverse mean curvature flow. If $\phi$ is
Lagrangian then $M^2$ is, up to dilations, the Clifford torus $\s^1 \times \s^1$.
\end{corollary}

We involve the sectional curvature in the next result in order to characterize the Example~\ref{Anciaux0} and the Example~\ref{ProductCircles} when $n=3$.

\begin{corollary}\label{ClosedLagrKDim3}
Let $\phi : M^3 \rightarrow \C^{3}$ be a closed orientable Lagrangian homothetic soliton for the inverse mean curvature flow. 
If the sectional curvature $K$ of $M^3$ verifies that $K\geq 0$ or $K\leq 0$ then
$M^3$ is, up to dilations, the Hopf immersion of \,$\s^1 \times \s^2$ in $\C^3$ described in the Example~\ref{Anciaux0} or the immersion of \,$\s^1 \times \s^1 \times \s^1$ in $\R^6$ described in the Example~\ref{ProductCircles}. 
\end{corollary}
\begin{proof}
Using Theorem \ref{ClosedLagr} we know that, up to dilations, $M=\s^1 \times N^{2}$, where $(N^{2},g)$ is a closed orientable manifold, and $\phi$ is given by
$
\phi(e^{it},x)=e^{it} \psi (x),
$
where $\psi:N^{2}\rightarrow\s^{5}\subset \C^3$ is a minimal Legendrian immersion. The induced metric $\phi^* \langle \cdot, \cdot \rangle$  is simply the product metric  $ dt^2 \times g$. 
In this case $\psi $ produces a minimal Lagrangian closed orientable surface in $\C P^2$ via the Hopf fibration with the same induced metric.  We now make use of some results in \cite{Ya74} on this class of minimal Lagrangian surfaces in $\C P^2$. If the genus of $N$ is zero, we can conclude that $\psi$ is totally geodesic and we obtain the Example~\ref{Anciaux0} for $n=3$. When the genus of $N$ is non null, since $K\geq 0$ or $K\leq 0$  necessarily $K_N\geq 0$ or $K_N\leq 0$ respectively. Then Theorem 7 of \cite{Ya74} implies that $K_N=0$ and 
$N$ must be the generalized Clifford torus in $\C P^2$. In this case we obtain the Example~\ref{ProductCircles} for $n=3$ (see Remark \ref{LegenLagr}).
\end{proof}

If we assume that the sectional curvature of a closed Lagrangian homothetic soliton for the IMCF is constant, we obtain the following uniqueness result for the product of $n$ circles in $\R^{2n}$.

\begin{corollary}\label{ClosedLagrKcte}
Let $\phi : M^n \rightarrow \C^{n}$ be a closed Lagrangian homothetic soliton for the inverse mean curvature flow. If the sectional curvature $K$ of $M^n$ is constant then $M^n$ is, up to dilations, the immersion of \,$\s^1 \times \stackrel{n)}{\dots} \times \s^1$ in $\R^{2n}$ described in the Example~\ref{ProductCircles}. 
\end{corollary}
\begin{proof}
Using again Theorem~\ref{ClosedLagr} we have that, up to dilations, $M=\s^1 \times N^{n-1}$, where $(N^{n-1},g)$ is a closed manifold, and $\phi$ is given by
$
\phi(e^{it},x)=e^{it} \psi (x),
$
where $\psi:N^{n-1}\rightarrow\s^{2n-1}\subset \C^n$ is a minimal Legendrian immersion. The induced metric $\phi^* \langle \cdot, \cdot \rangle$  is the product metric  $ dt^2 \times g$. Thus,
if we assume that $K$ is constant, necessarily $M$ must be flat, i.e.\ $K=0$. 
Then the sectional curvature $K_N$ of $N$ is zero too.

Via the Hopf fibration $\Pi$, $\psi$ provides a minimal Lagrangian in $\C P^{n-1}$ with zero sectional curvature. Using the main result of \cite{LZ94} we have that $\Pi\circ \psi$ is the generalized Clifford torus. Therefore, taking into account Remark \ref{LegenLagr}, $\phi $ is the Example~\ref{ProductCircles}.
\end{proof}

The following result follows the same spirit of Corollaries~\ref{Th IMRN14} and \ref{Cor IMRN14}. 

\begin{corollary}\label{ClosedLagrPinching}
Let $\phi : M^n \rightarrow \C^n$ be a closed Lagrangian homothetic soliton for the inverse mean
curvature flow. If 
\begin{equation}\label{pinch1}
\frac{|\sigma|^2}{|H|^2}\leq \frac{11n-6}{3n^2}
\end{equation}
either $|\sigma|^2 /|H|^2=(3n-2)/n^2$ and $\phi $ is given, up to dilations, by the Hopf immersion of \,$\s^1 \times \s^{n-1}$ in $\C^n$ described in Example~\ref{Anciaux0} or $n=3$, $|\sigma|^2 /|H|^2=1$ and $\phi $ is given, up to dilations, by \,$\s^1 \times \s^1 \times \s^1 \hookrightarrow \C^3$.
\end{corollary}
\begin{proof}
By Theorem~\ref{ClosedLagr} we can consider, up to dilations,  $\phi:\s^1 \times N^{n-1}\rightarrow\C^n$, 
$
\phi(e^{it},x)=e^{it} \psi (x),
$
with $\psi:N^{n-1}\rightarrow\s^{2n-1}\subset \C^n$ a minimal Legendrian immersion. Since $|\phi|=1$ we know that $|H|^2=n^2$ (see Remark~\ref{First examples}).
Via the Hopf fibration $\Pi$, $\psi$ provides a minimal Lagrangian immersion of $N$ in $\C P^{n-1}$ such that $|\sigma_{\Pi\circ\psi}|^2=|\sigma|^2-3n+2$, using the formula given in part 3 of Proposition~3 in \cite{RU98}. Then (\ref{pinch1}) gives that $|\sigma_{\Pi\circ\psi}|^2 \leq 2n/3$.
Now Theorem 1 of \cite{Xi92} implies that $\Pi\circ \psi$ is totally geodesic (and so $|\sigma|^2 =3n-2$) or the generalized Clifford torus of $\C P^2$. Taking into account Remark \ref{LegenLagr}, we conclude the proof.
\end{proof}

\begin{remark}
{\rm 
The proof of Corollary \ref{ClosedLagrPinching} implies that a closed Lagrangian homothetic soliton for the IMCF in $\C^n$ always satisfies that
\begin{equation}\label{pinch0}
\frac{|\sigma|^2}{|H|^2}\geq \frac{3n-2}{n^2}
\end{equation}
and the equality holds only for the Hopf immersion of \,$\s^1 \times \s^{n-1}$ in $\C^n$ described in Example~\ref{Anciaux0}. On the other hand, we easily can verify that the bound for $|\sigma|^2/|H|^2$ given in (\ref{pinch1}) is always greater than the one given in (\ref{boundAn}). Moreover, the bound given in (\ref{pinch0}) and (\ref{boundAn}) only coincide if $n=2$ with the Clifford torus.
}
\end{remark}

We finish this section characterizing the Example~\ref{Anciaux0} by its high degree of symmetry.

\begin{corollary}\label{ClosedLagrSO(n)}
Let $\phi : M^n \rightarrow \C^n$ be a closed Lagrangian homothetic soliton for the inverse mean
curvature flow. If $\phi $ is invariant under the action of the
special orthogonal group given by
\begin{equation}\label{eq:SOn}
\begin{array}{c}
SO(n)\times \C^{n}\longrightarrow \C^n \\
(A,(z_1,\dots,z_n)) \mapsto (z_1,\dots,z_n)A 
\end{array}
\end{equation}
then $\phi $ is given, up to dilations, by the Hopf immersion of \,$\s^1 \times \s^{n-1}$ in $\C^n$ given in Example~\ref{Anciaux0}.
\end{corollary}

\vspace{0.3cm}

\section{Lagrangian homothetic solitons for the IMCF}

The simplest Lagrangian submanifolds in complex Euclidean space $\C^n$ are plane curves, corresponding to $n=1$.
Taking into account Remark \ref{iff}, for a given plane regular curve $\alpha =\alpha (t)$, $t\in I \subseteq \R$, the equation (\ref{IMCFsol}) reduces to 
\begin{equation}\label{ICSFsol}
a \kappa \langle \alpha', J \alpha  \rangle =|\alpha'|
\end{equation}
where $J$ denotes the $+\pi/2$-rotation in $\C$ and $\kappa$ is the (signed) curvature of $\alpha$.
The homothetic solitons for the inverse curve shortening flow are the solutions of (\ref{ICSFsol}).
Andrews proved in \cite{And03} that circles centered at the origin are the only compact regular solutions of (\ref{ICSFsol}). The classification of all the homothetic soliton curves was given in \cite{DLW15}
and they are explicitly described using the slope $\theta $ of the tangent vector as parameter for $\alpha (\theta)= (x(\theta),y(\theta))$, where
\begin{eqnarray}\label{curves}
x(\theta)=-\nu'(\theta)\cos\theta - \nu(\theta) \sin \theta \\
y(\theta)=-\nu'(\theta)\sin\theta + \nu(\theta) \cos \theta \nonumber
\end{eqnarray}
with
\begin{equation}\label{nu}
\nu(\theta)= \left\{ 
\begin{array}{ll}
c_1 \cos (\sqrt{1-a}\,\theta)+ c_2 \sin (\sqrt{1-a}\, \theta) & {\rm if } \ a<1 \\
c_1+c_2 \, \theta & {\rm if } \ a=1 \\
c_1 \cosh (\sqrt{a-1}\, \theta) + c_2 \sinh (\sqrt{a-1}\, \theta) & {\rm if } \ a>1
\end{array}
\right.
\end{equation}
with $\theta\in \R$ and $c_1,c_2 \in \R$. 
\begin{remark}\label{ICSFcurves}
{\rm 
It is mentioned in \cite{DLW15} that when $a=1$, if $c_2=0$ we get the circle of radius $|c_1|$ and if $c_2 \neq 0$ it becomes the involute of the circle of radius $|c_2|$. In addition, when $a>1$, if $c_1=c_2=1$ the classical logarithmic spirals appear (see Figure \ref{InvoluteSpiral}).

We observe that when $a<1$ taking $a=1-m^2/k^2$, $m,k\in \N$, $m\neq k$, we get epicycloids (see Figure \ref{Epicycloids}) if $m<k $, i.e.\ $0<a<1$, and  hypocycloids (see Figure \ref{Hypocycloids}) if $m>k $ 
i.e.\ $a<0$. See also Figure \ref{together}.
}
\end{remark}

\begin{figure}[h!]
\begin{center}
\includegraphics[height=3cm]{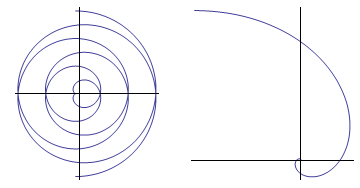}
\caption{Involute of a circle and a logarithmic spiral}
\label{InvoluteSpiral}
\end{center}
\end{figure}

\begin{figure}[h!]
\begin{center}
\includegraphics[height=3cm]{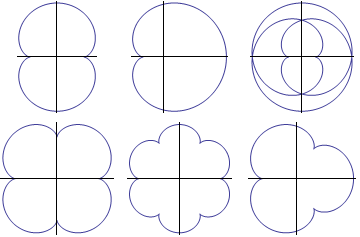}
\caption{Epicycloids}
\label{Epicycloids}
\end{center}
\end{figure}

\begin{figure}[h!]
\begin{center}
\includegraphics[height=3cm]{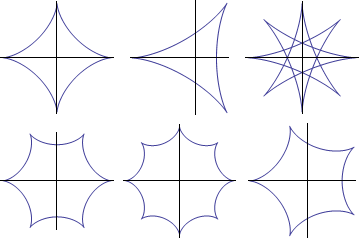}
\caption{Hypocycloids}
\label{Hypocycloids}
\end{center}
\end{figure}

\begin{figure}[h!]
\begin{center}
\includegraphics[height=5cm]{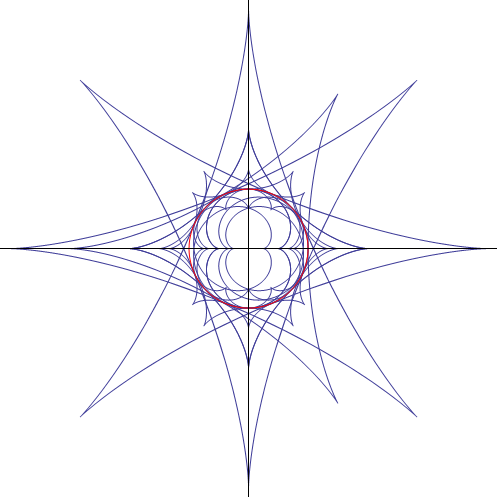}
\caption{Epicycloids, circle and hypocicloids together}
\label{together}
\end{center}
\end{figure}

Making the Cartesian product of $n$ plane curves we get flat Lagrangians in $\C^n$. The Example \ref{ProductCircles} is obtained using simply $n$ circles (of the same radius) centered at the origin. 
In the following result we look for all the Lagrangian homothetic solitons for the IMCF constructed in this way.

\begin{proposition}\label{LagrProduct}
The Lagrangian homothetic solitons for the inverse mean curvature flow with constant velocity $a\neq 0$
that are product of $n$-plane curves are given (up to dilations and congruences) by
\begin{enumerate}
\item $\s^1 \times \stackrel{n)}{\dots} \times \,\s^1 \hookrightarrow \C^{n}$, with $a=1/n$.
\item $\s^1 \times \stackrel{k)}{\dots} \times \,\s^1 \!\times \R^{n-k} \hookrightarrow \C^{n}$, $1\leq k \leq n-1$, with $a=1/k$.
\item $\mathcal C \times \R^{n-1}\hookrightarrow \C^{n}$, where $\mathcal C$ is a homothetic soliton for the inverse curve shortening flow, with $a\neq 0$ the constant velocity of $\mathcal C$.
\end{enumerate}
\end{proposition} 
\begin{proof}
Let $\alpha_1,\dots,\alpha_n$  be plane regular curves, and denote by $s_j$
the arc length parameter of $\alpha_j$, $1 \leq j \leq n$. Then the
product immersion
$$\phi(s_1, ... , s_n)=(\alpha_1(s_1), ... , \alpha_n(s_n))$$ is a
flat Lagrangian immersion whose mean curvature vector is 
\[
H = \sum_{j=1}^n \kappa_j J\phi_{s_j}
\]
where $\kappa_j$ denotes the curvature of the curve $\alpha_j$,
$1 \leq j \leq n$. On the other hand, it is easy to check that 
\[
\phi^\perp = -\sum_{j=1}^n  \langle \alpha_j ', J \alpha_j  \rangle J\phi_{s_j}.
\]
Thus we have that $\phi$ satisfies (\ref{IMCFsol}) if and only if 
\begin{equation}\label{SolitonLagrProduct}
\frac{\kappa_j}{\kappa_1^2+\dots+\kappa_n^2}=a \langle \alpha_j ', J \alpha_j \rangle, \ 1 \leq j \leq n .
\end{equation}
According to (\ref{SolitonLagrProduct}), $ \langle \alpha_j ', J \alpha_j \rangle=0 \Leftrightarrow \kappa_j=0$ since $a\neq 0$. But if $\langle \alpha_j ', J \alpha_j \rangle=0$  $\forall j$, $1 \leq j \leq n$, then $H=0$. So we can assume that $\langle \alpha_1 ', J \alpha_1 \rangle \neq 0$ and then (\ref{SolitonLagrProduct}) implies that $\kappa_i, \ 2 \leq i \leq n$, are constant. 

If $\kappa_i=0$  $\forall i$, $2 \leq i \leq n$, then (\ref{SolitonLagrProduct}) gives that $a\kappa_1 \langle \alpha_1 ', J \alpha_1 \rangle = 1$, which means that $\alpha_1$ is a soliton with constant velocity $a$ according to (\ref{ICSFsol}). In this case we arrive at $\mathcal C \times \R^{n-1}$, where $\mathcal C$ is a homothetic soliton for the inverse curve shortening flow with constant velocity $a\neq 0$.

Otherwise, we can take $\kappa_l=1/R_l$, $1\leq l \leq k$ and $\kappa_m=0$, $k+1 \leq m \leq n$, with $1\leq k \leq n$. Up to translations, we can write $\alpha_l(s_l)=R_l e^{i s_l/R_l}$, $1\leq l \leq k$, and now (\ref{SolitonLagrProduct}) gives 
\[
\frac{\kappa_l^2}{\kappa_1^2+\dots+\kappa_k^2}=a, \ 1\leq l \leq k
\]
This means that $R_1=\dots =R_k$ and, in addition, $a=1/k$. In this way, up to dilations, we get $\s^1 \times \stackrel{n)}{\dots} \times \,\s^1$ if $k=n$ and $\s^1 \times \stackrel{k)}{\dots} \times \,\s^1 \!\times \R^{n-k}$ if $k<n$.
\end{proof}
\vspace{0.2cm}

The Lagrangian expanders for the IMCF described in Theorem~\ref{ClosedLagr} belong to a bigger class of Lagrangian homothetic solitons that will be also studied in this section. We start recalling the notion of pseudoumbilical Lagrangians introduced in \cite{RU98} and \cite{CMU01}. A Lagrangian submanifold $M^n$ in $\C^n$ is called {\em pseudoumbilical} if there exists a non parallel closed conformal vector field $X$ on $M$ such that $\sigma (X,X)=\rho JX$ for some smooth function $\rho$. It was proved in Corollary~1 of \cite{RU98} that the pseudoumbilical Lagrangians in $\C^n$ are locally given by the Lagrangian immersions  
\[ \phi : I\times N\longrightarrow \C^n ,  \quad
\phi(t,x)  = \alpha(t)\psi(x)\]
where  $\alpha:I\rightarrow \C^*$ is a regular plane curve and $\psi:N\longrightarrow\s^{2n-1}\subset \C^n$
is a Legendrian immersion of a Riemannian $(n-1)$-manifold $(N,g)$.
The closed Lagrangian expanders for the IMCF described in Theorem~\ref{ClosedLagr} are obtained by considering the {\it circle} $\alpha (t)=e^{it}$ and a {\it minimal} Legendrian immersion $\psi$ of a closed $(n-1)$-manifold $N$.

The main result of this section provides the local classification of the Lagrangian homothetic solitons for the IMCF which are pseudoumbilical. 
\begin{theorem}\label{LagrPseudo}
Let $\phi : M^n \rightarrow \C^n$ be a Lagrangian pseudoumbilical  homothetic soliton for the inverse mean
curvature flow with constant velocity $a$. Then $\phi$ is locally congruent to 
$$
I \times N\longrightarrow \C^n , 
\quad (t,x) \mapsto \alpha(t)\psi(x)
$$
where $\psi:N\longrightarrow\s^{2n-1}\subset \C^n$
is a minimal Legendrian immersion of a Riemannian $(n-1)$-manifold $N$ and $\alpha:I\rightarrow \C^* $ is a regular plane curve such that $\alpha ^n$ is a homothetic soliton for the inverse curve shortening flow with constant velocity $n\,a$.
\end{theorem}
\begin{proof}
Thanks to the aforementioned local classification of pseudoumbilical Lagrangians, we must simply study when $\phi(t,x)=\alpha(t)\psi(x)$ verifies (\ref{IMCFsol}), with $\alpha $ a regular plane curve and $\psi:N\rightarrow\s^{2n-1}\subset \C^n$ a Legendrian immersion of a Riemannian $(n-1)$-manifold $(N,g)$.
In Theorem 4.1.\ of \cite{AC11} one can find a complete study of this type of Lagrangian immersions. Concretely we will use that the induced metric is 
\begin{equation}\label{metric}
\phi^* \langle \cdot, \cdot \rangle = |\alpha ' |^2 dt^2 \times |\alpha|^2  g,
\end{equation}
and the Lagrangian angle map of $\phi$ is given by
\begin{equation}\label{angle}
\theta_\phi (t,x) =G_\alpha(t)+ \theta_\psi(x) ,
\end{equation}
where $G_\alpha:= \arg \alpha' + (n-1) \arg \alpha$ and $\theta_\psi$
is the Legendrian angle map of $\psi$ (see for instance Section 2 in \cite{AC11} for background on Legendrian immersions). Using (\ref{metric}) and the Legendrian character of $\psi$, we get that
\begin{equation}\label{normalphi}
\phi^\perp = -\frac{\langle \alpha',J\alpha \rangle}{|\alpha'|^2} J\phi_t
\end{equation}
On the other hand, taking (\ref{Hlagr}), (\ref{metric}) and (\ref{angle}) into account, looking at (\ref{normalphi}) it is clear that for $H$ being collinear with $\phi^\perp$ it is necessary the minimality of $\psi $ and so
\begin{equation}\label{Hphi}
H = \frac{G'_\alpha}{|\alpha'|^2} J\phi_t
\end{equation}
Hence $\phi$ satisfies (\ref{IMCFsol}) if and only if $a \langle \alpha',J\alpha \rangle G'_\alpha = |\alpha'|^2$. As a summary, $\phi=\alpha\, \psi$ is a homothetic soliton for the IMCF with constant velocity $a$ if and only if $\psi$ is minimal and $\alpha$ satisfies
\begin{equation}\label{eq alpha}
a \langle \alpha',J\alpha \rangle \left(\arg \alpha' + (n-1) \arg \alpha \right)'=|\alpha'|^2 
\end{equation}
that is equivalent to 
\begin{equation}\label{eq alpha 2}
a \langle \alpha',J\alpha \rangle \left(|\alpha'|\kappa_\alpha + (n-1) \frac{\langle \alpha',J\alpha \rangle}{|\alpha|^2} \right)=|\alpha'|^2 
\end{equation}
where $\kappa_\alpha$ is the curvature of $\alpha$. But (\ref{eq alpha 2}) has a simpler expression by considering $\beta:= \alpha^n$. It is not difficult to check that (\ref{eq alpha 2}) is equivalent to 
\begin{equation}\label{eq beta}
n\, a \langle \beta',J\beta \rangle \kappa_\beta =|\beta'|
\end{equation}
Taking into account (\ref{ICSFsol}), we have that (\ref{eq beta}) implies that $\beta=\alpha ^n$ is a homothetic soliton for the one dimensional IMCF with constant velocity $n\,a$.
\end{proof}

If we take in Theorem~\ref{LagrPseudo} the totally geodesic Legendrian immersion $\psi:\s^{n-1}\rightarrow\s^{2n-1}\subset \C^n$, $\psi (x)=x$, we get $SO(n)$-invariant Lagrangian homothetic solitons for the IMCF that play the role of the equivariant Lagrangian self-similar solutions to the mean curvature flow in $\C^n$ studied in \cite{An06}.

\begin{corollary}\label{LagrRev}
Let $\phi : M^n \rightarrow \C^n$ be a Lagrangian homothetic soliton for the inverse mean
curvature flow. If $\phi $ is invariant under the action of the
special orthogonal group given in (\ref{eq:SOn}) then $\phi $ is, up to dilations, the immersion \,$\R \times \s^{n-1} \rightarrow \C^n$ given by
\[
(\theta,x) \mapsto |\alpha(\theta)|^{1/n}\, e^{i\arg \alpha (\theta) /n} (x_1,\dots,x_n)
\]
where $\alpha (\theta)= \left( -\vartheta'(\theta)\cos\theta - \vartheta(\theta) \sin \theta ,
-\vartheta'(\theta)\sin\theta + \vartheta(\theta) \cos \theta \right)$, with
\[
\vartheta(\theta)= \left\{ 
\begin{array}{ll}
c_1 \cos (\sqrt{1-a}\,\theta)+ c_2 \sin (\sqrt{1-a}\, \theta) & {\rm if } \ a<1/n \\
c_1+c_2 \, \theta & {\rm if } \ a=1/n \\
c_1 \cosh (\sqrt{a-1}\, \theta) + c_2 \sinh (\sqrt{a-1}\, \theta) & {\rm if } \ a>1/n
\end{array}
\right.
\]
and $c_1,c_2 \in \R$. 
\end{corollary}
According to Corollary \ref{ClosedLagrSO(n)}, the only closed one in the above family is the Hopf immersion described in Example \ref{Anciaux0}, which corresponds to $a=1/n$, $c_2=0$ and $|c_1|=n^{n/2}$.


\vspace{0.3cm}



\begin{thebibliography}{1}\bibliographystyle{alpha}

\bibitem[An06]{An06} H.~Anciaux.
\newblock {\em Construction of Lagrangian self-similar solutions to the mean curvature flow
in $\C^n$.}
\newblock  Geom.\ Dedicata {\bf 120} (2006), 37--48.

\bibitem[AC11]{AC11} H.~Anciaux and I.~Castro.
\newblock {\em Construction of Hamiltonian-minimal Lagrangian submanifolds in complex Euclidean space.}
\newblock  Results Math.\ {\bf 60} (2011), 325--349.

\bibitem[And03]{And03} B.~Andrews.
\newblock {\em Classifications of limiting shapes for isotropic curve flows.}
\newblock J.\ Amer.\ Math.\ Soc.\ {\bf 16} (2003), 443--459.

\bibitem[BG04]{BG04} V.~Borrelli and C.~Gorodski.
\newblock {\em Minimal Legendrian submanifolds of $\s^{2n+1}$ and absolutely area-minimizing cones.}
\newblock Diff.\ Geom.\ Appl.\  {\bf 21} (2004), 337--347.

\bibitem[Br13]{Br13} S.~Brendle.
\newblock {\em Embedded minimal tori in $\s^3$ and the Lawson conjecture.}
\newblock  Acta Math. {\bf 211} (2013), 177--190.

\bibitem[BN04]{BN04} H.~Bray and A.~Neves.
\newblock {\em Classification of prime 3-manifolds with Yamabe invariant greater than $\R P^3$.}
\newblock Ann.\ of Math.\ {\bf 159} (2004), 407--424.

\bibitem[CL14]{CL14} I.~Castro and A.M.~Lerma.
\newblock {\em The Clifford torus as a self-shrinker for the Lagrangian mean curvature flow.}
\newblock Int.\ Math.\ Res.\ Not.\ {\bf 2014} (2014), 1515--1527.

\bibitem[CLU06]{CLU06} I.~Castro, H. Li and F.~Urbano.
\newblock {\em Hamiltonian minimal Lagrangian submanifolds in complex space forms.}
\newblock Pacific J.\ Math.\ {\bf 227} (2006), 43--65.

\bibitem[CMU01]{CMU01} I.~Castro, C.R.Montealegre and F.~Urbano.
\newblock {\em Closed conformal vector fields and Lagrangian submanifolds in complex space forms.}
\newblock Pacific J.\ Math.\ {\bf 199} (2001), 269--302.

\bibitem[CU04]{CU04} I.~Castro and F.~Urbano.
\newblock {\em On a new construction of special Lagrangian immersions in complex Euclidean space.}
\newblock Quart.\ J.\ Math.\ {\bf 55} (2004), 253--265.

\bibitem[Ch13]{Ch13} B.-Y.~Chen.
\newblock {\em Classification of spherical Lagrangian submanifolds in complex Euclidean spaces.}
\newblock  Int.\ Electron.\ J.\ Geom.\ {\bf 6} (2013), 1--8.

\bibitem[CdCK78]{CdCK78} S.S.~Chern, M.~Do Carmo and S.~Kobayashi.
\newblock {\em Minimal submanifolds of sphere with second fundamental form of constant length.}
\newblock Shiing-Shen Chern Selected Papers, Springer-Verlag 1978, pp. 393--409.

\bibitem[DLW15]{DLW15} G.~Drugan, H.~Lee and G.~Wheeler.
\newblock {\em Solitons for the inverse mean curvature flow}
\newblock  arXiv: 1505.00183v1 [math.DG].

\bibitem[Ge90]{Ge90} C.~Gerhardt.
\newblock {\em Flow of nonconvex hypersurfaces into spheres.}
\newblock J.\ Differential Geom.\  {\bf 32} (1990), 299--314.

\bibitem[HL82]{HL82} R.~Harvey and H.B.~ Lawson.
\newblock {\em Calibrated geometries.}
\newblock  Acta Math.\  {\bf 148} (1982), 47--157.

\bibitem[HI97]{HI97} G.~Huisken and T.~Ilmanen.
\newblock {\em The Riemannian Penrose inequality.}
\newblock Int.\ Math.\ Res.\ Not.\ {\bf 20} (1997), 1045--1058.

\bibitem[HI01]{HI01} G.~Huisken and T.~Ilmanen.
\newblock {\em The  inverse mean curvature flow and the Riemannian Penrose inequality.}
\newblock J.\ Differential Geom.\  {\bf 59} (2001), 353--437.

\bibitem[HI08]{HI08} G.~Huisken and T.~Ilmanen.
\newblock {\em Higher regularity of the inverse mean curvature flow.}
\newblock J.\ Differential Geom.\  {\bf 80} (2008), 433--451.

\bibitem[La69]{La69} H.B.~Lawson.
\newblock {\em Local rigidity theorems for minimal hypersurfaces.}
\newblock  Ann.\ of Math.\ {\bf 89} (1969), 187--197.

\bibitem[La70]{La70} H.B.~Lawson.
\newblock {\em Complete minimal surfaces in $\s^3$.}
\newblock  Ann.\ of Math.\ {\bf 92} (1970), 335--374.

\bibitem[LZ94]{LZ94} A.M.~Li and G.~Zhao.
\newblock {\em Totally real submanifolds in $\C P^n$.}
\newblock  Arch.\ Math.\  {\bf 62} (1994), 562--568.

\bibitem[Na81]{Na81} H.~Naitoh.
\newblock {\em Totally real parallel submanifolds in $P^n(\C)$}
\newblock Tokyo J.\ Math.\ {\bf 4} (1981), 279--306.

\bibitem[Pi85]{Pi85} U.~Pinkall.
\newblock {\em Hopf tori in $\s^3$.}
\newblock  Invent.\ Math.\ {\bf 81} (1985), no 2, 379--386

\bibitem[RU98]{RU98} A.~Ros and F.~Urbano.
\newblock {\em Lagrangian submanifolds of $\C^n$ with conformal Maslov form and the Whitney sphere.}
\newblock J.\ Math.\ Soc.\ Japan {\bf 50} (1998), 203--226.

\bibitem[Si68]{Si68} J.~Simons.
\newblock {\em Minimal varieties in Riemannian manifolds.}
\newblock Ann.\ of Math.\ {\bf 88} (1968), 62--105.

\bibitem[Sm90]{Sm90} K.~Smoczyk.
\newblock {\em Remarks on the inverse mean curvature flow.}
\newblock Asian J.\ Math.\ {\bf 4} (2000), 331--336.

\bibitem[Ur90]{Ur90} J.~Urbas.
\newblock {\em On the expansion of starshaped hypersurfaces by symmetric functions of their principal curvatures.}
\newblock Math.\ Z.\ {\bf 205} (1990), 355--372.

\bibitem[Xi92]{Xi92} C.~Xia.
\newblock {\em On the minimal submanifolds in $\C P^m (c)$ and $\s^N (1)$.}
\newblock  Kodai Math.\ J.\ {\bf 15} (1992), 141--153.

\bibitem[Ya74]{Ya74} S.T.~Yau.
\newblock {\em Submanifolds with constant mean curvature I.}
\newblock Amer.\ J.\ Math.\ {\bf 96} (1974), 346--366.


\end{thebibliography}
\end{document}